\documentclass[reqno,centertags,12pt]{amsart}
\usepackage[letterpaper,margin=1.3in]{geometry}
\usepackage{amsmath,amsthm,amsfonts,amssymb,enumerate}
\usepackage[bookmarksopen=true]{hyperref}
\usepackage[nobysame,abbrev]{amsrefs}

\newcommand{\f}{\frac}

\newcommand{\no}{\notag}
\newcommand{\lb}{\label}
\newcommand{\bb}{\mathbb}
\newcommand{\cc}{\mathcal}
\newcommand{\ol}{\overline}

\newcommand{\bs}{\backslash}
\newcommand{\al}{\alpha}
\newcommand{\be}{\beta}
\newcommand{\de}{\delta}
\newcommand{\ga}{\gamma}

\newcommand{\la}{\lambda}
\newcommand{\te}{\theta}
\newcommand{\si}{\sigma}

\newcommand{\eps}{\varepsilon}

\newcommand{\pd}{\partial}
\newcommand{\ca}{C}
\newcommand{\x}{\mathbf x}
\newcommand{\supp}{\mathrm{supp}}

\newcommand{\<}{\langle}
\renewcommand{\>}{\rangle}
\renewcommand{\Re}{\mathrm{Re}}
\renewcommand{\Im}{\mathrm{Im}}

\newtheorem{theorem}{Theorem}

\theoremstyle{definition}

\theoremstyle{remark}

\numberwithin{equation}{section}
\numberwithin{theorem}{section}

\begin{document}

\title{Sharp lower bounds for the Widom factors on the real line}
\author{G\"{o}kalp Alpan}
\address{Department of Mathematics, Rice University, Houston, TX 77005, USA}
\email{alpan@rice.edu}

\author{Maxim Zinchenko$^{1}$}
\address{Department of Mathematics and Statistics, University of New Mexico, 311 Terrace Street NE, MSC01 1115, Albuquerque, NM 87106, USA}
\email{maxim@math.unm.edu}
\thanks{\footnotesize $^{1}$Research supported in part by Simons Foundation grant CGM-581256.}

\subjclass[2010]{Primary 41A17; Secondary 41A44, 42C05, 33C45, 47B36}
\keywords{Widom factors, Szeg\H{o} class, Equilibrium measure, Extremal polynomials, Jacobi polynomials, Jacobi matrices, Isospectral torus}

\begin{abstract}
We derive lower bounds for the $L^p(\mu)$ norms of monic extremal polynomials with respect to compactly supported probability measures $\mu$. We obtain a sharp universal lower bound for all $0<p<\infty$ and all measures in the Szeg\H{o} class and an improved lower bound on $L^2(\mu)$ norm for several classes of orthogonal polynomials including Jacobi polynomials, isospectral torus of a finite gap set and orthogonal polynomials with respect to the equilibrium measure of an arbitrary non-polar compact subset of $\bb R$.
\end{abstract}

\date{\today}
\maketitle

\section{Introduction}

Let $K$ be a non-polar compact subset of $\bb C$ and $\mu$ a probability Borel measure with $\supp(\mu)=K$. In this work we investigate lower bounds on the $L^p(\mu)$ norms of monic polynomials. A well known inequality that goes back to Szeg\H{o} \cite{Sze24} (for a textbook presentation see \cite[Theorem 5.5.4]{Ran95} or \cite[Theorem 5.7.8]{Sim11}) provides such a lower bound for $L^\infty(K)$ norm,
\begin{align}\label{ChebLB}
\|P_n\|_\infty = \sup_{z\in K}|P_n(z)| \ge \ca(K)^n, \quad P_n\in\cc P_n, \; n\in\bb N,
\end{align}
where $\cc P_n$ is the set of monic polynomials of degree $n$ and $\ca(K)$ denotes the logarithmic capacity of $K$. The inequality \eqref{ChebLB} is sharp in the class of subset of $\bb C$, however for compact sets $K\subset\bb R$, Schiefermayr \cite{Sch08} showed that the inequality can be improved to
\begin{align}\label{ChebLB2}
\|P_n\|_\infty \ge 2\ca(K)^n, \quad P_n\in\cc P_n, \; n\in\bb N,
\end{align}
which is optimal in the class of subsets of $\bb R$.

We are interested in finding sharp analogs of the above inequalities for $L^p(\mu)$ norms. To simplify the notation we introduce the Widom factors,
\begin{align}
W^p_n(\mu) = \f{\inf_{P_n\in\cc P_n}\|P_n\|_p}{C(K)^n}, \quad n\in\bb N, \; 0<p\le\infty,
\end{align}
where as usual $\|P_n\|_p=\left(\int|P_n(z)|^pd\mu(z)\right)^{1/p}$, $0<p<\infty$. Since $\mu$ is a probability measure, by H\"older's inequality, $W^p_n(\mu) \le W^q_n(\mu)\le W^\infty_n(K)$ for $0<p<q<\infty$.

From the application point of view there are two important cases $p=\infty$ and $p=2$. The monic polynomials that have the smallest $L^\infty(K)$ norm are known as the Chebyshev polynomials and those that have the smallest $L^2(\mu)$ norm are the orthogonal polynomials with respect to $\mu$. We use the term \textit{Widom factors} to commemorate the fundamental paper \cite{Wid69} where H.\ Widom studied asymptotics of the Chebyshev and orthogonal polynomials on sets $K$ consisting of a finite number of smooth Jordan curves and arcs. More recently, asymptotics and upper bounds on $W^\infty_n(K)$ have been studied in \cite{And17, AndNaz18, Eic17, GonHat15, CSZ17, CSYZ19, CSZ3, CSZ4, Tot09, Tot11, Tot12, Tot14, TotVar15,TotYud15}. Due to monotonicity of the Widom factors, an upper bound on $W^\infty_n(K)$ is automatically an upper bound for all $W^p_n(\mu)$. The main contribution of the present work is complementary sharp lower bounds for the Widom factors $W^p_n(\mu)$.

For absolutely continuous measures $\mu$ on the unit circle a lower bound and asymptotics of $W^2_n(\mu)$ date back to the work of Szeg\H{o} \cite{Sze20,Sze21} (for a textbook presentation  see \cite[Sections 2.2 and 2.3]{Sim05}). Asymptotics of $W^2_n(\mu)$ for more general measures and on other sets have been actively studied ever since \cite{AlpGon15, AlpGon16, Apt86, Bae11, Chr12, Chr19, CSZ11, DKS10, Ger60, KruSim15, Nev79, PehYud01, PehYud03, Sim05, Sim11, SimZla03, Sze22, Wid69}. In these works the central role is played by measures $\mu$ from the Szeg\H{o} class which in the most general setting is defined as follows. Given a probability measure $\mu$ with $K=\supp(\mu)$ a non-polar compact subset of $\bb C$, denote by $\mu_K$ the equilibrium measure of $K$ (see \cite{Ran95} or \cite[Section 5.5]{Sim11} for basic notions of logarithmic potential theory) and consider the Lebesgue decomposition of $\mu$ with respect to $\mu_K$, that is, $d\mu=fd\mu_K+d\mu_s$. The Szeg\H{o} class consists of such measures $\mu$ that have finite relative entropy with respect to $\mu_K$, that is,
\begin{align}
\int\log f(z)d\mu_K(z) > -\infty.
\end{align}
The relative entropy enters the asymptotics and lower bounds via the exponential relative entropy function
\begin{align}
S(\mu) = \exp\left[\int\log f(z)d\mu_K(z)\right].
\end{align}
As with the lower bound in the case $p=\infty$ (cf., \eqref{ChebLB} vs \eqref{ChebLB2}) there is a difference in the asymptotics of $W^2_n(\mu)$ depending on whether the measure $\mu$ is supported on $\bb R$ or on $\bb C$, for example,
\begin{align} \label{SzOPUC}
\lim_{n\to\infty}\left[W^2_n(\mu)\right]^2 = S(\mu)
\end{align}
for measures $\mu$ with $\supp(\mu)=\pd\bb D$ (see for example \cite[Theorem 2.3.1]{Sim05}) and
\begin{align} \label{SzOPRL}
\lim_{n\to\infty}\left[W^2_n(\mu)\right]^2 = 2S(\mu)
\end{align}
for measures $\mu$ with $\supp(\mu)=[-2,2]$ (see for example \cite{PehYud01}, \cite[Theorem 13.8.8]{Sim05}).

Recently a lower bound for the Widom factors $W^2_n(\mu)$ was obtained in \cite{Alp17} for the equilibrium measure $\mu=\mu_K$ of a general non-polar compact set $K\subset\bb R$ and in \cite{Alp19} for a general Szeg\H{o} class measure $\mu$ on $\bb C$,
\begin{align} \lb{AlpLB}
\left[W^2_n(\mu)\right]^2 \ge S(\mu).
\end{align}
The goal of the present work is to extend \eqref{AlpLB} to all Widom factors $W^p_n(\mu)$ and investigate to what extent such a lower bound is sharp and whether it can be improved for measures supported on $\bb R$ with a special emphasis on the case $p=2$.
It turns out that the lower bound \eqref{AlpLB} is sharp in the Szeg\H{o} class even for measures with $\supp(\mu)=[-2,2]$. The sharpness for measures on $\bb R$ is a rather surprising result as it stands in contrast with the case $p=\infty$ (cf., \eqref{ChebLB} vs \eqref{ChebLB2}) and the Szeg\H{o} asymptotics (cf., \eqref{SzOPUC} vs \eqref{SzOPRL}). Nevertheless, we will show that for several special classes of measures on $\bb R$ the lower bound can be improved. In particular, for the equilibrium measures $\mu=\mu_K$ of compact non-polar sets $K\subset\bb R$ the lower bound \eqref{AlpLB} improves by a factor of $2$,
\begin{align} \lb{LB2}
\left[W^2_n(\mu)\right]^2 \ge 2S(\mu).
\end{align}
In light of the asymptotics \eqref{SzOPRL}, the lower bound \eqref{LB2} is the best possible. We also obtain similar improvements on the lower bounds of Widom factors $W^p_n(\mu_K)$ for $p>1$. Besides the equilibrium measure we prove the optimal lower bound \eqref{LB2} for measures from the finite gap isospectral torus of half-line Jacobi matrices and for Jacobi weights $d\mu_{\al,\be}(x)=c_{\al,\be}(1-x)^\al(1+x)^\be\chi_{[-1,1]}(x)dx$ on $[-1,1]$ for a certain range of parameters $\al, \be$.

The plan of the paper is as follows. In Section 2, we extend the lower bound of \eqref{AlpLB} to the case of general Widom factors $W_n^p(\mu)$ and show that the bound is optimal not only in the class of measures on the complex plane but also on the real line. In particular, for measures on $\bb R$ the Szeg\H{o} condition alone is insufficient for \eqref{LB2}. In Section 3 we obtain increased lower bounds on $W^p_n(\mu_K)$ for the equilibrium measures $\mu_K$ on compact non-polar subsets of $\bb R$. In Section 4 we consider lower bounds on $W^2_n(\mu_{\al,\be})$ for the Jacobi weights over the full range of parameters. In Section 5, we prove \eqref{LB2} for measures $\mu$ associated with half-line Jacobi matrices from finite gap isospectral tori. Finally, in Section 6, we discuss some open problems.

\section{A sharp lower bound for the Widom factors}

In this section we extend the lower bound \eqref{AlpLB} to the general Widom factors and show that our lower bound is optimal in the class of Szeg\H{o} measures even if the support of the measure is an interval on the real line.

\begin{theorem}\label{ULBthm}
Let $0<p<\infty$ and $\mu$ be a Borel probability measure with $K=\supp(\mu)$ a non-polar compact subset of $\bb C$. Then
\begin{align}
\label{UniversalLB}
W_n^p(\mu) \ge S(\mu)^{1/p}, \quad n\in\bb N.
\end{align}
\end{theorem}
\begin{proof}
If $S(\mu)=0$ there is nothing to prove. Let us assume that $S(\mu)>0$. We modify the argument used in the proof of Theorem 1.2 in \cite{Alp19}.
Let $d\mu=f\mu_K+d\mu_s$ and write $P_n\in\cc P_n$ as $P_n(z)= \prod_{j=1}^n (z-z_j)$. Then
\begin{align}
\|P_n\|_p^p &= \left(\int |P_n|^p f\, d\mu_K+\int |P_n|^p d\mu_s\right) \geq \int |P_n|^p f\, d\mu_K \label{babba11}\\
&= \exp\left[\log{\left(\int |P_n|^p f\, d\mu_K\right)}\right]\label{babba}\\		
&\geq \exp\left[{\int \log{(|P_n|^p f)}\,d \mu_K }\right] \label{baba1} \\
&= \exp\left[\int \log{f}\, d\mu_K\right]
\exp\left[p \int \sum_{j=1 }^n \log |z-z_j| d\mu_K(z)\right] \label{baba2}\\
&\geq S(\mu) \ca(K)^{np} \label{baba3}.
\end{align}
	
Note that, \eqref{baba1} follows from Jensen's inequality and \eqref{baba3} follows from Frostman's theorem, see Theorem 3.3.4 (a) in \cite{Ran95}.  The inequality \eqref{UniversalLB} follows by taking $p$-th root and dividing by $\ca(K)^n$.
\end{proof}

It is easy to see that \eqref{UniversalLB} is sharp in the class of probability measures on the complex plane since for the equilibrium measure on the unit circle $\mu_{\pd\bb D}$ we have $1=S(\mu_{\pd\bb D})^{1/p}\le W_n^p(\mu)\le \|z^n\|_p=1$ for all $n\in\bb N$ and $0<p<\infty$. The next result shows that for $0<p<\infty$, $S(\mu)$ is the best possible lower bound for $W_n^p(\mu)$ in the Szeg\H{o} class of probability measures on the real line.

\begin{theorem}
For each $0<p<\infty$ and $n\in\bb N$ fixed,
\begin{align}
\label{ULB-Sharp}
\inf_{\mu} \big[W^p_n(\mu)\big]^p/S(\mu) = 1,
\end{align}
where the infimum is taken over probability measures on $K=[-2,2]$ with $S(\mu)>0$.
\end{theorem}
\begin{proof}
First, assume that $np\geq 1$. Let $N$ be the integer satisfying $np-1 < N \le np$ and consider the measures
\begin{align*}
d\mu_{\eps}(x) = c_{\eps}|x|^{N-np}\prod_{j=1}^N|x^2-j^2\eps^2|^{-1/2}d\mu_K(x), \quad \eps>0,
\end{align*}
where $c_{\eps}>0$ is the normalization constant chosen such that $\mu_{\eps}(K)=1$. The equilibrium measure is given by $d\mu_K(x)=\f1\pi\f{\chi_K(x)}{\sqrt{4-x^2}}dx$. Since $K$ is a regular set for potential theory, by Frostman's theorem, the logarithmic potential $U_K(z)=\int\log|x-z|d\mu_K(x)$ equals $\log\ca(K)=0$ for all $z\in K$, so for each $0<\eps<1/N$ we get
\begin{align}
\label{SLB-S}
S(\mu_{\eps}) &= \exp\bigg[ \int\log\Big(c_{\eps}|x|^{N-np}\prod_{j=1}^N|x^2-j^2\eps^2|^{-1/2}\Big)d\mu_K(x) \bigg] \no
\\&=
c_{\eps}\exp\bigg[(N-np)U_K(0)-\f12\sum_{j=1}^N U_K(j\eps)+U_K(-j\eps)\bigg]
= c_{\eps}.
\end{align}
On the other hand, we have
\begin{align}
\label{SLB-W}
W_n^p(\mu_\eps)^p \le \int|x^n|^p d\mu_{\eps}(x) = c_{\eps}\int\f{|x|^N}{\prod_{j=1}^N|x^2-j^2\eps^2|^{1/2}}d\mu_K(x) \to c_\eps
\end{align}
as $\eps\to0$ since
\begin{align*}
\f{|x|^N}{\prod_{j=1}^N|x^2-j^2\eps^2|^{1/2}} = \f{1}{\prod_{j=1}^N|1-(j\eps/x)^2|^{1/2}} \le \f{1}{|1-\f{N^2}{(N+1)^2}|^{N/2}} \text{ for } |x|\ge(N+1)\eps
\end{align*}
so by the dominated convergence theorem,
\begin{align*}
\int_{|x|\ge(N+1)\eps} \f{|x|^N}{\prod_{j=1}^N|x^2-j^2\eps^2|^{1/2}}d\mu_K(x)
\to \int d\mu_K = 1 \text{ as } \eps\to0
\end{align*}
and
\begin{align*}
&\int_{|x|<(N+1)\eps} \f{|x|^N}{\prod_{j=1}^N|x^2-j^2\eps^2|^{1/2}}d\mu_K(x) =
\int_{-1}^1\f{|t|^N}{\prod_{j=1}^N|t^2-\f{j^2}{(N+1)^2}|^{1/2}}
d\mu_K((N+1)\eps t)
\\
&\qquad
\le
\f1\pi\int_{-1}^1\f{|t|^N}{\prod_{j=1}^N|t^2-\f{j^2}{(N+1)^2}|^{1/2}}
\f{(N+1)\eps dt}{\sqrt{4-(N+1)^2\eps^2t^2}}
\to 0 \text{ as } \eps\to0.
\end{align*}
Thus, by \eqref{SLB-S} and \eqref{SLB-W}, $\limsup_{\eps\to0}W_n^p(\mu_{\eps})^p/S(\mu_{\eps})\le1$. This combined with \eqref{UniversalLB} yields \eqref{ULB-Sharp}.

Next, assume that $np<1$. Consider the measure $d\nu(x)= c |x|^{-np} d\mu_K(x)$ where $c>0$ is chosen so that $\nu(K)=1$. Then

\begin{align*}
S(\nu) &= \exp\bigg[ \int\log\Big(c|x|^{-np}\Big)d\mu_K(x) \bigg] \no
\\&=
c\exp[(-np)U_K(0)]
= c.
\end{align*}
We also have
\begin{align*}
W_n^p(\nu)^p \le \int|x^n|^p d\nu(x) = c\int d\mu_K(x) =c
\end{align*}
Thus, $W_n^p(\nu)^p/S(\nu)\le 1$. This combined with \eqref{UniversalLB} yields \eqref{ULB-Sharp} in this case.
\end{proof}

\section{Lower bounds for the equilibrium measures on subsets of $\bb R$}

In this section we improve the lower bound \eqref{UniversalLB} for equilibrium measures on general compact non-polar subsets of $\bb R$.

\begin{theorem} \label{EqLBthm}
Let $K\subset\bb R$ be a compact non-polar set. Then for each $p>1$,
\begin{align}
\label{EqLB1}
W_n^p(\mu_K) \ge 2\left(\f{(m!)^2}{(2m)!}\right)^{\f1{2m}} > S(\mu_K)^{1/p}=1, \quad n\in\bb N,
\end{align}
where $m=\big\lceil\f{p}{2(p-1)}\big\rceil$. In particular, for $p\ge2$,
\begin{align}
\label{EqLB2}
W_n^p(\mu_K) \ge \sqrt2, \quad n\in\bb N,
\end{align}
and the case $p=2$ is the improved lower bound \eqref{LB2} which is sharp in the class of equilibrium measures of non-polar compact subsets of $\bb R$.
\end{theorem}
\begin{proof}
First, note that \eqref{EqLB2} is a special case of \eqref{EqLB1}. Also note that  $m\ge p/[2(p-1)]$ which is equivalent to $p\ge 2m/(2m-1)$. Since $W_n^p(\mu_K)$ is nondecreasing with respect to $p$ it suffices to prove \eqref{EqLB1} for $p=2m/(2m-1)$, $m\in\bb N$.

Next, we prove \eqref{EqLB1} in the special case of a finite gap compact set $K\subset\bb R$. In this setting we recall the uniformization map for finite gap sets as discussed in \cite{CSZ10} or Sections 9.5--9.7 in \cite{Sim11}. The uniformization map is a unique conformal map $\x:\bb D\to\ol{\bb C}\bs K$ normalized by $\x(0)=\infty$ and $\lim_{z\to0}z\x(z)>0$. It is known that $\x$ is symmetric under complex conjugation, $\x(\bar z)=\ol{\x(z)}$, has an analytic extension to $\ol{\bb C}\bs\Lambda$, where $\Lambda\subset\pd\bb D$ is a certain null set, and $\x:\pd\bb D\bs\Lambda\to K$ preserves the equilibrium measure (cf., Corollary 4.6 in \cite{CSZ10} or Theorem 9.7.6 in \cite{Sim11}),
\begin{align}
\label{HarmMesInvar}
\int_K f(x)d\mu_K(x) = \int_0^{2\pi} f(\x(e^{i\te}))\f{d\te}{2\pi}, \quad f\in L^1(d\mu_K).
\end{align}
In the following we will also need the associated Blaschke product $B(z)$ which is the unique bounded analytic function on $\bb D$ with $|B(e^{i\te})|=1$ a.e.\ on $\pd\bb D$, zeros at $\x^{-1}(\infty)$, and normalized by $\lim_{z\to0}z^{-1}B(z)>0$. By Theorem 4.4 in \cite{CSZ10} or Theorem 9.7.5 in \cite{Sim11} the Blaschke product $B(z)$ has a connection to the Green function $G_K(z)$ of the domain $\ol{\bb C}\bs K$ via $|B(z)|=\exp[-G_K(\x(z))]$ and it satisfies (cf. (9.7.35) and (9.7.37) in \cite{Sim11})
\begin{align}
\label{xBcap}
\lim_{z\to0}\x(z)B(z)=\ca(K).
\end{align}

Now consider an arbitrary monic polynomial $P_n(x)$ of degree $n$, and let $Q_n(x)=\Re(P_n(x))$, $x\in\bb R$. Then $Q_n(x)$ is a monic polynomial with coefficients given by the real parts of the coefficients of $P_n(x)$, $Q_n(x)$ is real-valued on $K$, and satisfies $\|P_n\|_p\ge\|\Re(P_n)\|_p=\|Q_n\|_p$. In addition, $B(z)^nQ_n(\x(z))$ has only removable singularities on $\bb D$ and hence can be identified with a bounded analytic function with $\lim_{z\to0}B(z)^nQ_n(\x(z))=\ca(K)^n$ by \eqref{xBcap}. Thus,
\begin{align*}
C(K)^n &= \int_0^{2\pi} Q_n(\x(e^{i\te}))B(e^{i\te})^n \f{d\te}{2\pi}.
\end{align*}
Since the complex conjugation does not change the LHS and $Q_n(\x(e^{i\te}))$ we have
\begin{align*}
2C(K)^n &= \int_0^{2\pi} Q_n(\x(e^{i\te}))\big(B(e^{i\te})^n+\ol{B(e^{i\te})^n}\,\big) \f{d\te}{2\pi}.
\end{align*}
Applying H\"older's inequality, \eqref{HarmMesInvar}, and noting that $B(e^{i\te})^{-1}=\ol{B(e^{i\te})}$ we obtain
\begin{align}
2C(K)^n &\le \left[\int_0^{2\pi} |Q_n(\x(e^{i\te}))|^p \f{d\te}{2\pi}\right]^{\f1p}
\left[\int_0^{2\pi}\big(B(e^{i\te})^n+B(e^{i\te})^{-n}\big)^{2m} \f{d\te}{2\pi}\right]^{\f1{2m}}
\no \\
&= \left[\int_K|Q_n(x)|^p d\mu_K(x)\right]^{\f1p}
\left[\int_0^{2\pi}\sum_{j=0}^{2m}{2m\choose j}B(e^{i\te})^{2n(m-j)} \f{d\te}{2\pi}\right]^{\f1{2m}}
\no \\
&= \|Q_n\|_p\left[{2m\choose m}+2\Re\sum_{j=0}^{m-1}{2m\choose j}\int_0^{2\pi} B(e^{i\te})^{2n(m-j)} \f{d\te}{2\pi}\right]^{\f1{2m}}.
\end{align}
Since $\int_0^{2\pi}B(e^{i\te})^k \f{d\te}{2\pi} = B^k(0)=0$ for all $k\in\bb N$ and $\|Q_n\|_p\le \|P_n\|_p$ we get
\begin{align}
2C(K)^n &\le \|P_n\|_p {2m\choose m}^{\f1{2m}} = \|P_n\|_p \left(\f{(2m)!}{(m!)^2}\right)^{\f1{2m}}
\end{align}
which after rearranging yields \eqref{EqLB1} for finite gap sets $K\subset\bb R$.

Finally, we extend \eqref{EqLB1} to general non-polar compact sets $K\subset\bb R$ via an approximation argument of \cite{Alp17}. By Theorem 5.8.4 in \cite{Sim11} there exist finite gap sets $\{K_j\}_{j=1}^\infty$ such that $K\subset\dots\subset K_{j+1}\subset K_j\subset\dots\subset K_1\subset\bb R$, $K=\cap_{j=1}^\infty K_j$, $\ca(K_j)\to\ca(K)$, and $d\mu_{K_j}\to d\mu_K$ in the weak star sense as $j\to\infty$. Then for every monic polynomial $P_n(x)$ of degree $n$ we have by the finite gap lower bound that
\begin{align}
\|P_n\|_{p,\mu_K} &= \liminf_{j\to\infty}\|P_n\|_{p,\mu_{K_j}} \ge
\liminf_{j\to\infty} W_n^p(\mu_{K_j})\ca(K_j)^n \no
\\
&\ge
2\left(\f{(m!)^2}{(2m)!}\right)^{\f1{2m}} \liminf_{j\to\infty}\ca(K_j)^n
= 2\left(\f{(m!)^2}{(2m)!}\right)^{\f1{2m}} \ca(K)^n.
\end{align}
Dividing by $\ca(K)^n$ yields \eqref{EqLB1} for arbitrary non-polar compact set $K\subset\bb R$.

In the case $K=[-2,2]$ the orthogonal polynomials with respect to the equilibrium measure $\mu_K$ are the Chebyshev polynomials of the first kind and a straightforward computation shows that equality in \eqref{EqLB2} is attained for all $n\in\mathbb{N}$ proving that the lower bound \eqref{EqLB2} is sharp.
\end{proof}

\section{Lower bounds for the Jacobi weights}

Let $K=[-1,1]$ and consider the normalized Jacobi weights,
\begin{align}
d\mu_{\al,\be}(x) = c_{\al,\be}(1-x)^\al(1+x)^\be\chi_K(x)dx,
\end{align}
where $\al,\be>-1$ are parameters and $c_{\al,\be}$ is a normalization constant such that $\mu_{\al,\be}(K)=1$. We denote the corresponding monic orthogonal polynomials by $P_n^{\al,\be}$. By \cite[Section VII.1, Equation (25)]{Sue79},
\begin{align} \label{JacNorm}
\|P_n^{\al,\be}\|_2^2 = c_{\al,\be}\f{2^{\al+\be+2n+1}n!}{\al+\be+2n+1} \f{\Gamma(\al+n+1)\Gamma(\be+n+1)\Gamma(\al+\be+n+1)}{\Gamma(\al+\be+2n+1)^2}.
\end{align}
The equilibrium measure on $K=[-1,1]$ is given by $d\mu_K(x)=\f1\pi\f{\chi_K(x)}{\sqrt{1-x^2}}dx$ hence
$d\mu_{\al,\be}(x) = c_{\al,\be}\pi(1-x)^{\al+\f12}(1+x)^{\be+\f12}d\mu_K(x)$. Using Frostman's theorem and noting that $\ca(K)=\f12$ we get
\begin{align} \label{JacEntropy}
S(\mu_{\al,\be}) = c_{\al,\be}\pi\ca(K)^{\al+\be+1} = \f{c_{\al,\be}\pi}{2^{\al+\be+1}}.
\end{align}

Now, consider the ratios $R_n = \left[W_n^2(\mu_{\al,\be})\right]^2/S(\mu_{\al,\be})$. By \eqref{JacNorm} and \eqref{JacEntropy} we have
\begin{align}
R_n =
\f{2^{2\al+2\be+4n+2}n!}{\pi(\al+\be+2n+1)} \f{\Gamma(\al+n+1)\Gamma(\be+n+1)\Gamma(\al+\be+n+1)}{\Gamma(\al+\be+2n+1)^2}.
\end{align}
The optimal constant in the lower bound for $\left[W^2_n(\mu_{\al,\be})\right]^2$ is given by $\inf_n R_n$. Thus, we are interested in finding the parameters for which $\inf_n R_n$ is maximal. While estimating $\inf_n R_n$ directly is difficult, we can find values of the parameters $\al,\be$ so that the sequence $R_n$ is decreasing. In this case the improved lower bound \eqref{LB2} follows from the Szeg\H{o} asymptotics \eqref{SzOPRL}. In the other extreme, if $R_n$ is strictly increasing then, by \eqref{SzOPRL}, the optimal constant in the lower bound is strictly less than $2$ and is given by $R_1$.

Define the quantities
\begin{align}
D_n = \f{R_{n+1}}{R_n} - 1, \quad n\in\bb N.
\end{align}
Then the sequence $R_n$ is decreasing if and only if $D_n\le0$ for all $n$ and $R_n$ is strictly increasing if and only if $D_n>0$ for all $n$.
Using the identity $\Gamma(x+1)=x\Gamma(x)$ and introducing $s_n=\al+\be+2(n+1)$ we obtain
\begin{align}
D_n &= \f{16(n+1)(\al+n+1)(\be+n+1)(\al+\be+n+1)} {(\al+\be+2n+1)(\al+\be+2n+2)^2(\al+\be+2n+3)} - 1
\no \\
&= \f{[s_n^2-(\al+\be)^2][s_n^2-(\al-\be)^2]}{s_n^2(s_n^2-1)} - 1
\label{JacD1}
\\
&= \f{(\al^2-\be^2)^2+s_n^2[1-2(\al^2+\be^2)]}{s_n^2(s_n^2-1)}.
\label{JacD2}
\end{align}
Then for all $n\in\bb N$, we have $D_n\le0$ if $|\al|+|\be|\ge1$ by \eqref{JacD1} and $D_n>0$ if $\al^2+\be^2\le\f12$ and $|\al|+|\be|<1$ by \eqref{JacD2}. Since $s_n\to\infty$ as $n\to\infty$ it is clear from \eqref{JacD2} that when $\al^2+\be^2>\f12$ there exists $n_0$ such that $D_n<0$ for all $n\ge n_0$ and $D_n\geq 0$ for $1\leq n \leq n_0-1$ provided that $n_0>1$. Since $\lim_{n\to\infty}R_n=2$ by \eqref{SzOPRL}, it follows that for all values of $\al,\be$ the infimum of $R_n$ is equal to $\min\{2,R_1\}$. In addition, since $s_n>2$ we can estimate $D_n$ for $\al^2+\be^2\ge\f12$ by
\begin{align}
D_n \le \f{(\al^2+\be^2)^2+4[1-2(\al^2+\be^2)]}{s_n^2(s_n^2 -1)}
\end{align}
which implies that $D_n<0$ if $4-2\sqrt3<\al^2+\be^2<4+2\sqrt3$. Combining  these special cases we get the following result:
\begin{theorem}
For all $\al,\be>-1$ we have
\begin{align}
\left[W^2_n(\mu_{\al,\be})\right]^2 \ge L_{\al,\be}\, S(\mu_{\al,\be}), \quad n\in\bb N,
\end{align}
with the optimal constant $L_{\al,\be}$ given by
\begin{align}
L_{\al,\be} = \min\left\{2,\f{2^{2\al+2\be+6}}{\pi(\al+\be+3)} \f{\Gamma(\al+2)\Gamma(\be+2)\Gamma(\al+\be+2)}{\Gamma(\al+\be+3)^2}\right\}.
\end{align}
In addition, if $\al^2+\be^2\le\f12$ and $|\al|+|\be|<1$ then $L_{\al,\be}<2$ and if either $|\al|+|\be|\ge1$ or $\al^2+\be^2>4-2\sqrt3\approx0.536$ then $L_{\al,\be}=2$, that is, \eqref{LB2} holds. In particular, in the symmetric case $\al=\be$ the lower bound \eqref{LB2} holds if and only if $|\al|\ge\f12$.
\end{theorem}

\section{Lower bounds for measures from the isospectral tori}

For a finite gap set
\begin{align}
\label{FGSet}
K=[\al_1,\be_1]\cup\dots\cup[\al_{\ell+1},\be_{\ell+1}]
\end{align}
with $\al_1<\be_1<\al_2<\dots<\al_{\ell+1}<\be_{\ell+1}$, the isospectral torus $\cc T_K$ consists of two sided Jacobi matrices $J=\{a_k,b_k\}_{k=-\infty}^\infty$ with the spectrum $\si(J)=K$ which are reflectionless on $K$, that is, the diagonal Green functions $G_{n,n}(z)=\<\de_n,(J-z)^{-1}\de_n\>$, $n\in\bb Z$, of $J$ have purely imaginary boundary values a.e.\ on $K$, see for example Sections 5.13 and 7.5 in \cite{Sim11}. By Craig's formula (cf., Theorem 5.4.19 in \cite{Sim11}), the diagonal Green functions $G_{n,n}(z)$ of reflectionless Jacobi matrices are of the form
\begin{align}
\label{CraigFormula}
G_{n,n}(z) = -\prod_{j=1}^\ell(z-\ga_{n,j}) \left[\prod_{j=1}^{\ell+1}(z-\al_j)(z-\be_j)\right]^{-1/2}, \quad n\in\bb Z,
\end{align}
where $\ga_{n,j}\in[\be_j,\al_{j+1}]$, $j=1,\dots,\ell$, $n\in\bb Z$.

In this section we investigate the Widom factors for the spectral measure $\mu_n$ of the one-sided truncation $J_n=\{a_{n+k},b_{n+k}\}_{k=1}^\infty$ of $J\in\cc T_K$. Alternatively, such one-sided Jacobi matrices $J_n$ are characterized by the property of the associated $m$-function $m_n(z)=\<\de_1,(J_n-z)^{-1},\de_1\>$ being a minimal Herglotz function on the two sheeted Riemann surface with branch cuts along $K$ (cf., Theorems 5.13.10, 5.13.12, and 7.5.1 in \cite{Sim11}). The minimal Herglotz functions are characterized by Theorem 5.13.2 in \cite{Sim11} which implies that the spectral measures $\mu_n$ consist of an absolutely continuous component on $K$ and a finite number of mass points at the discrete eigenvalues of $J_n$, $\si_d(J_n)\subset\bb R\bs K$ (cf., (5.13.19), (5.13.24), (5.13.25) in \cite{Sim11}),
\begin{align}
d\mu_n(x) = \f1\pi\Im[m_n(x+i0)]\chi_K(x)dx + \sum_{\la\in\si_d(J)}\mathrm{res}_{z=\la}[m_n(z)]d\de_\la(x).
\end{align}
There is a connection between $G_{n,n}$ and $m_n$ obtained in the proof of Theorem 5.13.12 in \cite{Sim11},
\begin{align}
\label{mGrelation}
\Im[a_n^2m_n(x+i0)]=\f12\Im[-G_{n,n}(x+i0)^{-1}] \text{ for a.e. } x\in K,
\end{align}
and the zeros of $G_{n,n}(z)$ correspond to the poles of $m_n$ on either the first or the second sheet of the Riemann surface, hence $\si_d(J_n)$ is a subset of $\{\ga_{n,j}\}_{j=1}^\ell$, the zero set of $G_{n,n}$. Thus, using \eqref{CraigFormula}, \eqref{mGrelation}, and (5.13.24), (5.13.25) in \cite{Sim11} we get an explicit form of $\mu_n$,
\begin{align}
\label{IsoMeas1}
d\mu_n(x) =& \f1{2a_n^2\pi} \f{\sqrt{\prod_{j=1}^{\ell+1}|x-\al_j||x-\be_j|}} {\prod_{j=1}^\ell|x-\ga_{n,j}|} \chi_K(x)dx \no
\\ &+
\sum_{k:\ga_{n,k}\in\si_d(J_n)}\f1{a_n^2} \f{\sqrt{\prod_{j=1}^{\ell+1}|x-\al_j||x-\be_j|}} {\prod_{j=1,j\neq k}^\ell|\ga_{n,k}-\ga_{n,j}|}d\de_{\ga_{n,k}}(x).
\end{align}
By Theorem 5.5.22 and $(5.4.96)$ in \cite{Sim11}), the equilibrium measure $\mu_K$ of a finite gap set $K$ is given by
\begin{align}
\label{EqMes}
d\mu_K(x) = \f1\pi\f{\prod_{j=1}^\ell|x-c_j|} {\sqrt{\prod_{j=1}^{\ell+1}|x-\al_j||x-\be_j|}}\chi_K(x)dx,
\end{align}
where $c_j\in(\be_j,\al_{j+1})$, $j=1,\dots,\ell$, are the critical points of the Green function $G_K(z)$ for the domain $\ol{\bb C}\bs K$ with a logarithmic pole at infinity. Combining \eqref{IsoMeas1} and \eqref{EqMes} then gives the Lebesgue decomposition of $\mu_n$ with respect to $\mu_K$,
\begin{align}
\label{IsoMeas2}
d\mu_n(x) =& \f{1}{2a_n^2} \f{\prod_{j=1}^{\ell+1}|x-\al_j||x-\be_j|}
{\prod_{j=1}^\ell|x-c_j||x-\ga_{n,j}|}d\mu_K(x) \no
\\ &+
\sum_{k=1}^\ell\f{s_{n,k}}{a_n^2} \f{\sqrt{\prod_{j=1}^{\ell+1}|x-\al_j||x-\be_j|}} {\prod_{j=1,j\neq k}^\ell|\ga_{n,k}-\ga_{n,j}|}d\de_{\ga_{n,k}}(x),
\end{align}
where $s_{n,k}=1$ if $\ga_{n,k}\in\si_d(J_n)$ and $s_{n,k}=0$ otherwise. The factor $a_n^{-2}$ plays a role of the normalization constant and hence is uniquely determined by $\{\ga_{n,j},s_{n,k}\}_{j=1}^\ell$. By Theorems 5.13.5 and 7.5.1 in \cite{Sim11}, the class of such measures $\mu_0$ as $J$ runs through $\cc T_K$ consists of all possible choices of $\ga_{0,j}\in[\be_j,\al_{j+1}]$ and $s_{0,j}\in\{0,1\}$ with $s_{0,k}=0$ if $\ga_{0,j}$ is at an edge $\be_j$ or $\al_{j+1}$, $j=1,\dots,\ell$.

\begin{theorem}
Let $K\subset\bb R$ be a finite gap set and $\mu_0$ be the spectral measure of a half-line truncation $J_0$ of $J\in\cc T_K$, that is, $\mu_0$ is of the form \eqref{IsoMeas2}. Then
\begin{align}
\label{IsoLB}
\left[W_n^2(\mu_0)\right]^2 \ge 2 E(\mu_0)^2 S(\mu_0), \quad n\in\bb N,
\end{align}
where $E(\mu_0)$ is the eigenvalue function given by
\begin{align}
E(\mu_0) = \exp\Bigg[\sum_{x\in\supp(\mu_0)\bs K}G_K(x)\Bigg].
\end{align}
In particular, since $E(\mu_0)\ge1$, the improved lower bound \eqref{LB2} holds.
\end{theorem}
\begin{proof}
The proof will be based on the step-by-step sum rule of \cite{CSZ11}. Let $\mu_n$ denote the spectral measure of the one-sided truncation $J_n$ of $J$, $n\ge1$.
Then, by Theorem 4.2 in \cite{CSZ11} or Proposition 9.10.5 in \cite{Sim11}, we have
\begin{align}
\label{SumRule}
W_n^2(\mu_0) = \f{a_1\cdots a_n}{\ca(K)^n} = \f{E(\mu_0)S(\mu_0)^{1/2}}{E(\mu_n)S(\mu_n)^{1/2}}.
\end{align}
Since in each gap of $K$ the Green function $G_K(x)$ is positive and attains its maximal value at the critical points we have the estimate
\begin{align}
\label{Ebound}
1 \le E(\mu_n) \le \exp\left[\sum_{j=1}^{\ell}G_K(\ga_{n,j})\right].
\end{align}
Recalling that $G_K(z)=-\log\ca(K)+\int\log|z-x|d\mu_K(x)$, we get from \eqref{IsoMeas2},
\begin{align}
S(\mu_n) &= \f{1}{2a_n^2}\exp\left[\int\log\left(\f{\prod_{j=1}^{\ell+1}|x-\al_j||x-\be_j|} {\prod_{j=1}^{\ell}|x-c_j||x-\ga_{n,j}|}\right)d\mu_K(x)\right] \no
\\&=
\f{\ca(K)^{2}}{2a_n^2}\exp\left[-\sum_{j=1}^{\ell}G_K(c_j) - \sum_{j=1}^{\ell}G_K(\ga_{n,j})\right],
\end{align}
and hence,
\begin{align}
\label{ESbound}
E(\mu_n)^2S(\mu_n) \le \f{\ca(K)^2}{2a_n^2}\exp\left[\sum_{j=1}^{\ell}\big[G_K(\ga_{n,j})-G_K(c_j)\big]\right] \le \f{\ca(K)^2}{2a_n^2}.
\end{align}
Squaring \eqref{SumRule} and using \eqref{ESbound} give
\begin{align}
\f{a_1^2\cdots a_n^2}{\ca(K)^{2n}} = W_n^2(\mu_0)^2 = \f{E(\mu_0)^2S(\mu_0)}{E(\mu_n)^2S(\mu_n)} \ge \f{2a_n^2}{\ca(K)^2} E(\mu_0)^2S(\mu_0).
\end{align}
Cancelling $a_n^2/\ca(K)^2$ term and utilizing \eqref{Ebound} we obtain
$$
\f{a_1^2\cdots a_{n-1}^2}{\ca(K)^{2(n-1)}} = W_{n-1}^2(\mu_0)^2 \ge 2E(\mu_0)^2S(\mu_0), \quad n\in\bb N.
$$
\end{proof}

\section{Open problems}

\textit{Problem 1.} In Theorem \ref{EqLBthm}, the sharp lower bound for $W_n^p(\mu_K)^p$ is obtained for $p=2$. The sharp lower bound for $W_n^p(\mu_K)^p$ when $p\neq 2$ and $K\subset \mathbb{R}$ is an open problem. At least, we have a natural candidate for this lower bound:
It is known that on an interval $K=[-1,1]$ the monic Chebyshev polynomials of the first kind minimize $L^p(\mu_K)$ norms for all $1\leq p\le\infty$ (see for example p.\ 96 in \cite{Riv90}), hence the corresponding Widom factors can be evaluated explicitly in this case,
\begin{align}\label{super}
\big[W_n^p(\mu_K)\big]^p= \frac{2^p}{\pi} \int_{0}^{\pi} | \cos{\theta}|^p\, d\theta= \frac{2^p}{\sqrt\pi} \frac{\Gamma(\f{p+1}2)}{\Gamma(\f{p}2+1)}, \quad n\in\bb N.
\end{align}
Note that the right hand side of \eqref{super} is independent of $n$. When $p=2$ \eqref{super} gives the sharp lower bound \eqref{LB2} and the limit as $p\to\infty$ of the $p$-th root of \eqref{super} gives the sharp lower bound \eqref{ChebLB2}. We conjecture that
\begin{align*}
\big[W_n^p(\mu_K)\big]^p \geq  \frac{2^p}{\sqrt\pi} \frac{\Gamma(\f{p+1}2)}{\Gamma(\f{p}2+1)}, \quad n\in\bb N,
\end{align*}
when $K$ is a non-polar compact subset of $\mathbb{R}$ and $1\leq p <\infty$.

\textit{Problem 2.} Let $K$ be a finite gap set. Besides the equilibrium measure $\mu_K$ and measures from the isospectral torus of $K$, an important class of measures is the class of reflectionless measures. These are the measures appearing in the Herglotz representation of $G_{n,n}$ from \eqref{CraigFormula}, that is, given by $d\mu_{n,n}(x)=\f1\pi\Im[G_{n,n}(x+i0)]\chi_K(x)dx$. The equilibrium measure $\mu_K$ is a member of this class. We conjecture that \eqref{LB2} holds for all reflectionless measures on a finite gap set.

\textit{Problem 3.} Is there a simple characterization of Szeg\H{o} class measures on a finite gap set or even an interval for which \eqref{LB2} holds?

\textit{Problem 4.}
If $K$ is a finite gap set and $\mu$ be a Borel probability measure which is purely singular continuous with respect to $\mu_K$ and $\mathrm{supp}(\mu)=K$, then $W_n^2(\mu)\rightarrow 0$ since $S(\mu)=0$ by Theorem 4.5 in \cite{CSZ11}.

If $K_1=\overline{\mathbb{D}}$ and $\mu_1$ is the normalized area measure on $K_1$, then $P_n(z)= z^n$ is the $n$-th monic orthogonal polynomial with respect to $\mu_1$ and a straightforward calculation shows that $[W_n^2(\mu_1)]^2=\frac{1}{n+1}$. Since $\mu_{K_1}$ is the normalized arc-measure on the unit circle, $\mu_1$ is purely singular continuous with respect to $\mu_{K_1}$ and we have $W_n^2(\mu_1)\rightarrow 0$. It is also true that Widom factors for the normalized area measure on Jordan domains with analytic boundary goes to $0$, see Theorem 4.1
in \cite{Gustaf}.

If $K_2$ is the Cantor ternary set and $\mu_2$ is the Cantor measure, then $\mu_2$ is purely singular continuous with respect to $\mu_{K_2}$ by \cite{maka}. However, in this case it was conjectured in \cite[Conjecture 3.2]{KruSim15} that $\liminf{W_n^2}(\mu_2)>0$ based on numerical evidence.

It would be interesting to develop the theory of Widom factors for purely singular continuous measures (w.r.t.\ the equilibrium measure of the support).
Proving or disproving existence of such a measure $\mu$ satisfying the condition $\liminf{W_n^2}(\mu)>0$ would be a good start.




\begin{thebibliography}{99}

\bibitem{Alp17} G.\ Alpan, \emph{Orthogonal Polynomials Associated with Equilibrium Measures on  $\mathbb{R}$}, Potential Anal.\ \textbf{46} (2017), 393--401.

\bibitem{Alp19} G.\ Alpan \emph{Szeg\H{o}'s condition on compact subsets of $\mathbb{C}$}, J. Approx. Theory.\ \textbf{245} (2019), 130--136.

\bibitem{AlpGon15} G.\ Alpan, A.\ Goncharov, \emph{Widom factors for the Hilbert norm}, Banach Center Publ.\ \textbf{107} (2015), 11--18.

\bibitem{AlpGon16} G. Alpan, A.\ Goncharov, \emph{Orthogonal polynomials for the weakly equilibrium Cantor sets}, Proc.\ Amer.\ Math.\ Soc., \textbf{144} (2016), 3781--3795.

\bibitem{And17} V.\ V.\ Andrievskii, \emph{On Chebyshev polynomials in the complex plane}, Acta Math. Hungar., \textbf{152} (2017), 505--524.

\bibitem{AndNaz18} V.\ Andrievskii, F.\ Nazarov, \emph{On the Totik--Widom Property for a Quasidisk}, Constr. Approx. (2018), https://doi.org/10.1007/s00365-018-9452-4.

\bibitem{Apt86} A.I.\ Aptekarev, \emph{Asymptotic properties of polynomials orthogonal on a system of contours, and periodic motions of Toda chains}, Math.\ USSR Sb. \textbf{53} (1986), 233--260; Russian original in Mat.\ Sb. (N.S.) \textbf{125(167)} (1984), 231--258.

\bibitem{Bae11} A.\ Baernstein II, R.\ S.\ Laugesen, I.\ E.\ Pritsker, \emph{Moment inequalities for equilibrium measures in the plane}, Pure\ Appl.\ Math.\ Q., \textbf{7} (2011), 47--82

\bibitem{Chr12} J.\ S.\ Christiansen, \emph{Szeg\H{o}'s theorem on Parreau-Widom sets}, Adv.\ Math.\ \textbf{229}, (2012), 1180--1204.

\bibitem{Chr19} J.\ S.\ Christiansen, \emph{Dynamics in the Szeg\H{o} class and polynomial asymptotics}, J.\ Anal.\ Math.\ \textbf{137} (2019), no.\ 2, 723--749.

\bibitem{CSYZ19} J.\ S.\ Christiansen, B.\ Simon, P.\ Yuditskii, and M.\ Zinchenko, \emph{Asymptotics of Chebyshev Polynomials, II. DCT subsets of $\bb R$}, Duke Math. J. \textbf{168} (2019), 325--349.

\bibitem{CSZ10} J.\ S.\ Christiansen, B.\ Simon, and M.\ Zinchenko, \emph{Finite gap Jacobi matrices, I. The isospectral torus}, Constr.\ Approx.\ \textbf{32} (2010), no.\ 1, 1--65.

\bibitem{CSZ11} J.\ S.\ Christiansen, B.\ Simon, and M.\ Zinchenko, \emph{Finite gap Jacobi matrices, II.\ The Szeg\H{o} class}, Constr.\ Approx. \textbf{33} (2011), 365--403.

\bibitem{CSZ17} J.\ S.\ Christiansen, B.\ Simon, and M.\ Zinchenko, \emph{Asymptotics of Chebyshev Polynomials, I. Subsets of $\bb R$}, Invent. Math. \textbf{208} (2017), 217--245.

\bibitem{CSZ3} J.\ S.\ Christiansen, B.\ Simon, and M.\ Zinchenko, \emph{Asymptotics of Chebyshev Polynomials, III. Sets Saturating Szeg\H{o}, Schiefermayr, and Totik--Widom Bounds}, to appear in \emph{Analysis as a Tool in Mathematical Physics -- in Memory of Boris Pavlov}, ed. P. Kurasov, A. Laptev, S. Naboko and B. Simon, to be published by Birkhauser.

\bibitem{CSZ4} J.\ S.\ Christiansen, B.\ Simon, and M.\ Zinchenko, \emph{Asymptotics of Chebyshev Polynomials,\\ IV. Comments on the Complex Case}, to appear in J.\ Anal.\ Math.

\bibitem{DKS10} D.\ Damanik, R.\ Killip, and B.\ Simon, \emph{Perturbations of orthogonal polynomials with periodic recursion coefficients}, Annals of Math.\ (2) \textbf{171} (2010), 1931--2010.


\bibitem{Eic17} B.\ Eichinger, \emph{Szeg\H{o}--Widom asymptotics of Chebyshev polynomials on circular arcs}, J.\ Approx.\ Theory \textbf{217} (2017), 15--25.

\bibitem{Ger60} Ya.\ L.\ Geronimus, \emph{Polynomials Orthogonal on a Circle and Interval}, Pergamon Press, New York, 1960.

\bibitem{GonHat15} A. Goncharov and B. Hatino\u{g}lu, \emph{Widom factors}, Potential Anal.\ \textbf{42} (2015), 671--680.

\bibitem{Gustaf} B.\ Gustafsson, M.\ Putinar, E.\ B.\ Saff, N.\ Stylianopoulos, \emph{ Bergman polynomials on an archipelago: Estimates, zeros and shape reconstruction}, Adv.\ Math.\ \textbf{222}, 1405--1460.

\bibitem{KruSim15} H.\ Kr{\"u}ger, B.\ Simon, \emph{Cantor polynomials and some related classes of OPRL}, J.\ Approx.\ Theory.\ {\bf191} (2015), 71-–93

\bibitem{maka} N.\ Makarov, A,\ Volberg,\emph{On the harmonic measure of discontinuous fractals}, preprint, V.A. Steklov Math. Institute (Leningrad branch), E-6-86, Leningrad.

\bibitem{Nev79} P.~Nevai, \emph{Orthogonal polynomials}, Mem.\ Amer.\ Math.\ Soc. \textbf{18} (1979), no. 213, 1--183.

\bibitem{PehYud01} F.\ Peherstorfer and P.\ Yuditskii, \emph{Asymptotics of orthonormal polynomials in the presence of a  denumerable set of mass points}, Proc.\ Amer.\ Math.\ Soc. \textbf{129} (2001), 3213--3220.

\bibitem{PehYud03} F.\ Peherstorfer and P.\ Yuditskii, \emph{Asymptotic behavior of polynomials orthonormal on a homogeneous set}, J.\ Anal.\ Math. {\bf 89} (2003), 113--154.

\bibitem{Ran95} T.\ Ransford, \emph{Potential Theory in the Complex Plane}, Cambridge University Press, 1995.

\bibitem{Riv90} T.\ J.\ Rivlin, \emph{The Chebyshev Polynomials. From Approximation Theory to Algebra and Number Theory}, 2nd ed., Pure Appl. Math. (N.Y.), Wiley, New York 1990.

\bibitem{Sch08} K.\ Schiefermayr, \emph{A lower bound for the minimum deviation of the Chebyshev polynomial on a compact real set}, East J. Approx. \textbf{14} (2008), 223--233.

\bibitem{Sim05} B.\ Simon, \emph{Orthogonal Polynomials on the Unit Circle, Part 1: Classical Theory, Part 2: Spectral Theory}, AMS Colloquium Publication Series, Vol.\ 54, Providence, RI, 2005.

\bibitem{Sim11} B. Simon, \emph{Szeg\H{o}'s Theorem and Its Descendants: Spectral Theory for  $L^2$ Perturbations of Orthogonal Polynomials}, Princeton University Press, Princeton, 2011.

\bibitem{SimZla03} B.\ Simon and A.\ Zlato\v{s}, \emph{Sum rules and the Szeg\H{o} condition for orthogonal polynomials on the real line}, Comm.\ Math.\ Phys. \textbf{242} (2003), 393--423.



\bibitem{Sze20} G.\ Szeg\H{o}, \emph{Beitr\"{a}ge zur Theorie der Toeplitzschen Formen}, Math.\ Z.\ \textbf{6} (1920), 167--202.

\bibitem{Sze21} G.\ Szeg\H{o}, \emph{Beitr\"{a}ge zur Theorie der Toeplitzschen Formen, II}, Math.\ Z.\ \textbf{9} (1921), 167--190.

\bibitem{Sze22} G.\ Szeg\H{o}, \emph{\"Uber den asymptotischen Ausdruck von Polynomen, die durch eine Orthogonalit\"atseigenschaft definiert sind}, Math.\ Ann.\ \textbf{86} (1922), 114--139.

\bibitem{Sze24} G.\ Szeg\H{o}, \emph{Bemerkungen zu einer Arbeit von Herrn M. Fekete: \"Uber die Verteilung der Wurzeln bei gewissen algebraischen Gleichungen mit ganzzahligen Koeffizienten}, Math. Z. \textbf{21} (1924), 203--208.

\bibitem{Sue79} P.\ K.\ Suetin, \emph{Classical Orthogonal Polynomials}, 2nd ed., Nauka, Moscow, 1979.

\bibitem{Tot09} V.\ Totik, \emph{Chebyshev constants and the inheritance problem}, J.\ Approx.\ Theory \textbf{160} (2009), 187--201.

\bibitem{Tot11} V.\ Totik, \emph{The norm of minimal polynomials on several intervals}, J.\ Approx.\ Theory \textbf{163} (2011), 738--746.

\bibitem{Tot12} V.\ Totik, \emph{Chebyshev polynomials on a system of curves}, J.\ Anal.\ Math.\ \textbf{118} (2012), 317--338.

\bibitem{Tot14} V.\ Totik, \emph{Chebyshev polynomials on compact sets}, Potential Anal.\ \textbf{40} (2014), 511--524.

\bibitem{TotVar15} V.\ Totik and T.\ Varga, \emph{Chebyshev and fast decreasing polynomials}, Proc.\ Lond.\ Math.\ Soc.\ (3) \textbf{110} (2015), no.\ 5, 1057--1098.
\bibitem{TotYud15} V.\ Totik and P.\ Yuditskii, \emph{On a conjecture of Widom}, J.\ Approx.\ Theory \textbf{190} (2015), 50--61.

\bibitem{Wid69} H. Widom, \emph{Extremal polynomials associated with a system of curves in the complex plane}, Adv.\ in Math. \textbf{3} (1969), 127--232.

\end{thebibliography}
\end{document}